\documentclass[a4paper,11pt]{article}
\usepackage{ae} 
\usepackage[T1]{fontenc}
\usepackage[ansinew]{inputenc}
\usepackage[leqno,namelimits]{amsmath}
\usepackage{amsthm}
\usepackage{amssymb}
\usepackage{graphicx}

\newtheorem{thm}{Theorem}
\newtheorem{cor}[thm]{Corollary}

\newtheorem{prop}[thm]{Proposition}

\numberwithin{equation}{section}

 \DeclareMathOperator{\RE}{Re}
 \DeclareMathOperator{\IM}{Im}
 \newcommand{\eps}{\varepsilon}
 \newcommand{\To}{\longrightarrow}
 
 \newcommand{\s}{\mathcal{S}}
 
 \newcommand{\M}{\mathcal{M}}
 \newcommand{\J}{\mathcal{J}}
 \newcommand{\K}{\mathcal{K}}

 \newcommand{\Complex}{\mathbb{C}}

 \newcommand{\G}{\mathcal{G}}
 \newcommand{\Lom}{\mathcal{L}}
 
 \newcommand{\abs}[1]{\left\vert#1\right\vert}
 \newcommand{\set}[1]{\left\{#1\right\}}
 \newcommand{\seq}[1]{\left<#1\right>}
 \newcommand{\norm}[1]{\left\Vert#1\right\Vert}

\author{ Aik. Aretaki and J. Maroulas\footnote{Department of Mathematics, National\, Technical \,University \,of
Athens, Zografou Campus, Athens 15780, Greece. E-mail address:
maroulas@math.ntua.gr.}}
\title{Investigating the Numerical Range \\of Non Square Matrices}
\begin{document}
\maketitle
{\small\textbf{Abstract.}
A presentation \,of \,numerical\, range for rectangular matrices

is\, undertaken\,\,\,in \,this \,paper, introducing\, two \,different\, definitions\,\,\,and

elaborating basic properties. Then we are extended to the treatment \,of

rank-k numerical range.}\\\\
{\small\textbf{Key words}: numerical range, projectors, matrix norms, singular values.}



\section{Introduction}
Let $\M_{m,n}(\Complex)$ be the set of matrices $A=[a_{ij}]_{i,j=1}^{m,n}$
with entries $a_{ij}\in \Complex$. For $m=n$, the  set
\begin{equation}\label{w1}
F(A)=\set{ x^{*}Ax : x\in \Complex^{n}, \norm x_{2} = 1}
\end{equation}
is the well known \textit{numerical range} or \textit{field of
values} of $A$, for which basic pro\-perties can be found in
~\cite{H.J.T}, \cite{Rao} and ~\cite[chapter 22]{Halmos}. Equivalently, we say
that $F(A)=f(\mathcal{S}_{n})$, where $\mathcal{S}_{n}$ is the unit sphere of $\Complex^{n}$
and the function $f$ on $\mathcal{S}_{n}$ is defined by the bilinear mapping
$g:\mathcal{S}_{n}\times\mathcal{S}_{n}\rightarrow\Complex$, such that $f(x)=g(x,x)=x^{*}Ax$.
It is remarkable that $F(A)$ is closed and convex set and contains the set of eigenvalues
of $A$.

For $m\neq n$, the motivation herein is to investigate
\textit{''how the numerical range w(A) can be defined for a rectangular matrix
$A$''} based on the inner product and to develop some basic and
fundamental properties. As we may see, the results vary  and the approach is undertaken in two ways,
firstly we consider a natural extension of \eqref{w1} and on the
other hand, introducing the idea of restriction or extension of
dimensions of $A$, we are led to the relationship of $w(A)$ with the numerical range
of square matrices via projection matrices. Hence, generalizing  the
notion of definition \eqref{w1}, we consider the bilinear mapping
$g:\mathcal{S}_{n}\times\mathcal{S}_{m}\rightarrow\Complex$, $g(x,y)=y^{*}Ax$, which gives
rise to the numerical range of $m\times n $ matrix $A$, as the set
\begin{equation}\label{w2}
w(A)=\set{ y^{*}Ax : x\in \Complex^{n}, y\in \Complex^{m}, \norm
x_{2} = \norm y_{2} = 1}
\end{equation}
which is equal to $g(\mathcal{S}_{n}\times\mathcal{S}_{m})$. Note that for $m>n$ we have
\begin{eqnarray*}
F(\begin{bmatrix}
    A & 0_{m\times(m-n)} \\ \end{bmatrix}) & = &\set{ y^{*}Ax : y=\begin{bmatrix} x \\ \omega \\ \end{bmatrix}
    \in \Complex^{m}, x\in \Complex^{n}, \norm y_{2} = 1} \\
                                           & = & \norm x_{2}\set{ y^{*}A\frac{x}{\norm x_{2}} : y=\begin{bmatrix} x \\ \omega \\ \end{bmatrix}
    \in \Complex^{m}, x\in \Complex^{n}, \norm y_{2} = 1}
\end{eqnarray*}
\quad\quad\quad\quad\quad\quad\quad\quad\quad\,$\subseteq \,\,\,\, w(A).$

Proceeding, it is proved that $w(A)$ is identified with the circular
disc $\set{z\in \Complex : \abs z \leq \norm A _{2}}$,\,\, since \,the\,
unit\, vectors\,\, $x$ and $y$\, belong \,to different\,\, dimensional spaces. An
approximation of $w(A)$ from within, following, is shown,
assuming that the vectors $x, y$ in \eqref{w2} belong to subspaces
$\mathcal{F}\subset\Complex^{n}$ and
$\mathcal{G}\subset\Complex^{m}$, respectively. Recently, has been proposed
~\cite{Psarrakos} as numerical range of $A\in \M_{m,n}$ \textit{with
respect to matrix} $B\in\M_{m,n}$ the compact and convex set
\begin{equation}\label{w3}
w_{\norm{\cdot}}(A,B)=\bigcap_{z_{0}\in \Complex}{\set{z\in\Complex
: \abs{z-z_{0}}\leq\norm{A-z_{0}B}}}.
\end{equation}
The \eqref{w3} is an extension of definition of $F(A)$ for square
matrices in \cite{Bonsall} and clearly the numerical range, as in
\cite{Bonsall}, \cite{Bon-Dun-II}, is based on the notion of matrix
norm. In \cite{Psarrakos} has been proved that $w_{\norm{\cdot}}(A,B)$ coincides with the disc
\begin{equation}\label{w4}
\set{z\in\Complex : |z-\frac{\seq{A,B}}{\norm B^{2}}|\leq
\|{A-\frac{\seq{A,B}}{\norm B^{2}}B\|\sqrt{1-\norm B^{-2}}}}
\end{equation}
when $\norm B\geq1$ and the matrix norm $\norm \cdot$ is induced by
the inner product $\seq{\cdot,\cdot}$. The complicated formulation
of numerical range of $A$ and the necessity of independence of
$w_{\norm{\cdot}}(A,B)$ by the matrix $B$ in \eqref{w3} and \eqref{w4},
are signified in section 2.

Another proposal for the definition of numerical range for
rectangular matrices, which will be further exploited in section 3,
is the projection onto the lower or the higher dimensional subspace.
Let $m>n$ and the vectors $v_{1}, \ldots, v_{n}$ of $\Complex^{m}$
be orthonormal basis of $\Complex^{n}$. Clearly, the matrix
$P=HH^{*}$, where $H=\left[
                              \begin{array}{ccc}
                                v_{1} & \ldots & v_{n} \\
                              \end{array}
                            \right]$, is an orthogonal projector of
$ \Complex^{m}\To\Complex^{n} $. In this case, for $A\in\M_{m,n}$,
we define \emph{with respect to} $H$:
\begin{equation}\label{w5}
w_{l}(A)=F(H^{*}A)
\end{equation}
where obviously $H^{*}A$ is $n\times n$ matrix. Moreover, the
vector $y=Hx\in\Complex^{m}$ is projected onto $\Complex^{n}$ along
$\mathcal{K}$, where $\mathcal{K}$ is any direct complement of
$\Complex^{n}$, i.e. $\Complex^{m}=\Complex^{n}\oplus\mathcal{K}$.
Since, $\norm y =(x^{*}H^{*}Hx)^{1/2}=\norm x$, instead of
\eqref{w5}, it can also be provided a treatment of the numerical
range $w_{h}(A)$ of higher dimensional $m\times m$ matrix $AH^{*}$,
namely,
\begin{equation}\label{w6}
w_{h}(A)=F(AH^{*}).
\end{equation}
Similarly, if $m<n$, then $x=Hy$ and consequently
\begin{equation}\label{w7}
w_{l}(A)=F(AH),\quad  w_{h}(A)=F(HA).
\end{equation}

Apparently, by \eqref{w5}-\eqref{w7}, the numerical range of
$A\in\M_{m,n}$ via the projection of unit vectors onto
$\Complex^{n}$ or $\Complex^{m}$ is referred to the numerical range
of square matrix, indicating obviously the convexity of $w_{l}(A)$
and $w_{h}(A)$. Clearly, for $m=n$ and $H=I$, $w_{l}(A)$ and $w_{h}(A)$ are reduced to the
classical numerical range $F(A)$ in \eqref{w1}. In \eqref{w5} and \eqref{w6}, if $A$ is orthonormal
$(A^{*}A=I_{n})$, for $H=A$, clearly
\[
w_{h}(A)=[0,\sigma_{\max}(A)]=[0,1], \quad w_{l}(A)=[\sigma_{\min}(A), \sigma_{\max}(A)]=\set1
\]
where $\sigma_{\max}(\cdot)$ and $\sigma_{\min}(\cdot)$ \,denote the \,maximum\, and \,minimum
singular values of matrix. Some additional properties of these sets are exposed in section
3 including the notion of sharp point.

Accepted the definition \eqref{w2}, an equivalent representation of $w(A)$ in \eqref{w2} is
\begin{equation}\label{w8}\begin{split}
 w(A)=&\set{ z\in\Complex : PAQ = zS, \,\,where \,\,
P=yy^{*}, Q=xx^{*}, S=yx^{*}\right.\\
&\quad\quad\quad\quad\quad\quad\quad\quad\quad\left. \,\,and\,\,
x\in \Complex^{n}, y\in \Complex^{m}, \norm x_{2} = \norm y_{2} = 1}
\end{split}\end{equation}
In \eqref{w8} the matrices $P, Q$ are rank-1 orthogonal projections
of $\Complex ^{m}$ and $\Complex ^{n}$ and $S$ satisfies the
equation $PXQ=X$. In this way, in section 4 we are led to the
generalization of \textit{rank-k numerical range} for square
matrices
\begin{equation}\label{w9}\begin{split}
 \Lambda_{k}(A)&=\set{\lambda\in\Complex : PAP=\lambda P\,\,for\,\,some\,\,rank-k\right.\\
 &\quad\quad\quad\quad\quad\quad\quad\quad\quad\quad\,\,\left.orthogonal\,\,projection\,\,P}
\end{split}\end{equation}
which has been presented and extensively studied by Choi
\textit{et} \textit{al} in \cite{Choi},\cite{C-H-K-Z}, \cite{C-K-Z-q}, \cite{C-K-Z}
and later by other researchers in \cite{Hugo}, \cite{Li-Sze}, \cite{Poon-Li-Sze} and \cite{L-R}.
In this paper, for the $m\times n$ matrix $A$ and a
positive integer $k\geq 1$, \textit{the rank-k numerical range of
$A$} is defined by the set
\begin{equation}\label{w10}\begin{split}
\!\!\phi_{k}(A)&=\set{ z\in \Complex : PAQ = zS, \,for\,\, some\,\, rank-k\,\,orthogonal\right. \\
&\quad\quad\quad\quad\quad\quad\quad\quad\quad\quad\left.projections\,\,\ P\,\, and\,\, Q\,\,and\,\, S=PSQ}
\end{split}\end{equation}
For $m=n$ and $P=Q$, $\phi_{k}(A)=\Lambda_{k}(A)$.

These sets satisfy, analogously to $\Lambda_{k}(A)$, the inclusion
relationship
\[
w(A)=\phi_{1}(A)\supseteq \phi_{2}(A)\supseteq \ldots
\supseteq \phi_{\tau}(A)
\]
where $\tau=\min{\set{m,n}}$ and the proof is established in the final section
4. Then it is proved that $\phi_{k}(A)$, under constraints for the index $k$, is a circular ring
or a disc, presenting the non-emptiness and the convexity of the set for special cases.


\section{Properties of $w(A)$}

Recalling the definition of $w(A)$ in \eqref{w2}, we readily
recognize the property
$$w(kA)=kw(A).$$
The convexity of $w(A)$ is confirmed indirectly by the next statement.

\begin{prop}
For each $m\times n$ matrix $A$, $w(A)=\set{z \in\Complex : \abs
z\leq \norm A_{2}}$.
\end{prop}
\begin{proof}
Let $m>n$. Since the rows $\tilde{a}_{1}, \ldots, \tilde{a}_{n}$
of $A$ are linear dependent, we consider the unit vector $y_{0}$
such that $y_{0}^{*}A=0$. Then, for a unit vector $x$, we have
$y_{0}^{*}Ax=0$, i.e. $0\in w(A)$. Also, due to Cauchy-Schwarz
inequality, we obtain
\[
\abs{y^{*}Ax} = \abs{\seq{Ax,y}} \leq \norm{Ax}_{2} \norm y_{2} =
\norm{Ax}_{2} \leq \max_{\norm x_{2}=1}{\norm{Ax}_{2}}=\norm A_{2} =
\sigma_{\max}(A).
\]
If $z=re^{i\theta} \in \set{z : \abs z \leq \norm A_{2}}$, obviously
$0<r\leq \norm A_{2}$. Evenly, there exists a unit vector $\hat{x}$
such that $\norm{A\hat{x}}_{2}=r$, since the function $f(x)=
\norm{Ax}_{2} : \s_{n} \To (0,\norm A_{2}]$ is continuous, where
$\s_{n}$ is the compact unit sphere of $\Complex^{n}$. Thus, for
$\hat{y}=A\hat{x}/(\norm{A\hat{x}}_{2}e^{i\theta}) $, clearly $\norm
{\hat{y}}_{2}=1$ and
\[
\hat{y}^{*}A\hat{x} =
\frac{\hat{x}^{*}A^{*}A\hat{x}}{\norm{A\hat{x}}_{2}e^{-i\theta}}=\norm{A\hat{x}}_{2}e^{i\theta}=re^{i\theta}=z.
\]

Moreover, the boundary $\partial w(A)=\set{z: \abs z=\norm A_{2}}$
is attained, since, by the unit eigenvectors
$A^{*}Ax=\sigma_{\max}^{2}x$ and $AA^{*}y=\sigma_{\max}^{2}y$, we
receive the point
$\abs{y^{*}Ax}=\abs{y^{*}(\sigma_{\max}y)}=\sigma_{\max}$. Due to the fact that the
singular values of $A$ and $e^{i\theta}A$ are identical, the points
$y^{*}(e^{i\theta}A)x$ are also boundary points of the circular disc
$\set{z : \abs z\leq \norm A_{2}}$.
\end{proof}

We remark that the proof is not specially simplified if we consider the singular value decomposition of $A$
and the invariant under unitary equivalence.

\begin{cor}
Let $A\in\M_{m,n}$ and $z=y^{*}Ax\in w(A)$. Then we have
\begin{description}
  \item[ \,\,I.]
  $z\in F(\begin{bmatrix}
            0 & 2A \\
            0_{n\times m} & 0 \\
          \end{bmatrix})$  and corresponds to unit
vector $\omega=\frac{1}{\sqrt{2}}\begin{bmatrix}
                                   y \\
                                   x \\
                                 \end{bmatrix}$,
  \item[ II.] $w(A)=\bigcap_{0\leq\theta\leq2\pi}{\set{half\,\, plane\,\,\,:
  e^{-i\theta}\set{z: \RE
  z\leq\sigma_{\max}(A)}}},$
 \item[III.]
 if $A=\textbf{a}\in\Complex^{n}$,
  $w(\textbf{a})=\mathcal{D}(0,\norm{\textbf{a}}_{2})$.
\end{description}
\end{cor}
\begin{proof}
\textbf{I.} By Proposition 1, we have $\RE
z\in [-\sigma_{\max}(A),\sigma_{\max}(A)]$ and after some algebraic
manipulations, we obtain $\RE z=\omega^{*}\begin{bmatrix}
                                               0 & A \\
                                               A^{*} & 0 \\
                                             \end{bmatrix}\omega$, where
$\omega=\frac{1}{\sqrt{2}}\begin{bmatrix}
                            y \\
                            x \\
                          \end{bmatrix}$.
Similarly, $\IM z=\omega^{*}\begin{bmatrix}
                              0 & -iA \\
                              iA^{*} & 0 \\
                            \end{bmatrix}\omega$, i.e. $\IM z\in[-i\sigma_{\max}(A),i\sigma_{\max}(A)]$
and consequently $z=\omega^{*}\begin{bmatrix}
                            0 & 2A \\
                            0 & 0 \\
                          \end{bmatrix}\omega$.\\
\textbf{II.} The graph of $\partial w(A)$ is
constructed only by the values $e^{i\theta}\sigma_{\max}(A)$.\\
\textbf{III.} For $A=\textbf{\emph{a}}\in\Complex^{n}$, the unique
singular value of $\textbf{\emph{a}}$ is
$\sigma=\norm{\textbf{\emph{a}}}_{2}$.
\end{proof}

\begin{cor}
Let $A, B\in\M_{m,n}$, then holds:
\begin{description}
  \item[\,\,\,\,\textbf{I.}]
  $w(A)=w(A^{*})$,
  \item[\,\,\textbf{II.}]
  $w(\hat{A})\subseteq w(A)$, for any $p\times q$ submatrix $\hat{A}$ of $A$,
  \item[\textbf{III.}]
  $w(diag(A,B))=\max{\set{w(A),w(B)}}$, where $A\in\M_{m,n}$, $B\in\M_{n,m}$,
  \item[\textbf{IV.}]
  $w(A+B)\subseteq w(A)+w(B)$,
  \item[\,\,\textbf{V.}]
  $w(U^{*}AV)=w(A)$, where $U\in\M_{m}, V\in\M_{n}$ are unitary matrices.
\end{description}
\end{cor}
\begin{proof}
Statement (I) is an immediate consequence of
Proposition 1 and\, (II)\, is implied \,using\, the \,inequality
$\|\hat{A}\|_{2}\leq\norm A_{2}$\, (\cite{H.J.T}, Cor. 3.1.3,
p.149), where \,\,$\hat{A}$ is $p\times q$\, submatrix  of\,
$A$. Following,\, assertion \,(III) can be \,deduced from the
condition $\norm{diag(A,B)}_{2}=\max{\set{\norm A_{2},\norm B_{2}}}$
and for (IV), (V) the triangle inequality and the
unitarily invariant property of $\norm{\cdot}_{2}$ are applied,
respe\-cti\-ve\-ly.
\end{proof}

The computation of $w(A)$ from inside is presented by the next
proposition.

\begin{prop}
Let $l, k$ be positive integers less than $m, n$, respectively. Then
\begin{equation}\label{pr1}
w(A)=\mathcal{D}\left(0,\max_{
\substack{\xi_{1},\ldots,\xi_{l}\in\Complex^{m}\\
              \eta_{1},\ldots,\eta_{k}\in\Complex^{n}}}
{\norm{\begin{bmatrix} \xi_{i}^{*}A\eta_{j}  \\ \end{bmatrix}_{i,j=1}^{l,k}}}_{2}\right)
\end{equation}
where $\set{\xi_{1},\dots,\xi_{l}}$ and
$\set{\eta_{1},\ldots,\eta_{k}}$ are orthonormal vectors of
$\Complex^{m}$ and $\Complex^{n}$, respectively.
\end{prop}
\begin{proof}
Any vectors $x\in\Complex^{n}$ and
$y\in\Complex^{m}$ belong to subspaces
$\mathcal{F}\subseteq\Complex^{n}$ and
$\mathcal{G}\subseteq\Complex^{m}$. If
$\set{\eta_{1},\ldots,\eta_{k}}$ and $\set{\xi_{1},\ldots,\xi_{l}}$
are orthonormal bases of $\mathcal{F}$ and $\mathcal{G}$,
respectively, then
\[
x=\begin{bmatrix}
      \eta_{1} & \ldots & \eta_{k} \\
    \end{bmatrix}u,\,\,\,y=\begin{bmatrix}
                       \xi_{1} & \ldots & \xi_{l} \\
                     \end{bmatrix}v
\]
where $u\in\Complex^{k}$ and $v\in\Complex^{l}$. Since $\norm
x=\norm y=1$, $u$ and $v$ are also unit vectors and we have
\[
y^{*}Ax=v^{*}\begin{bmatrix}
                 \xi_{1}^{*} \\
                 \vdots \\
                 \xi_{l}^{*} \\
               \end{bmatrix}A\begin{bmatrix}
                         \eta_{1} & \ldots & \eta_{k} \\
                       \end{bmatrix}u=v^{*}\begin{bmatrix}
                                             \xi_{i}^{*}A\eta_{j} \\
                                           \end{bmatrix}_{i,j=1}^{l,k}u.
\]
Thus, we verify $\norm{\begin{bmatrix} \xi_{i}^{*}A\eta_{j}  \\
                    \end{bmatrix}_{i,j=1}^{l,k}}_{2}\leq \norm A_{2}$
and then the equation \eqref{pr1}.
\end{proof}

Note that in \eqref{pr1} for $k=n$ and $\begin{bmatrix}
                                            \eta_{1} & \ldots & \eta_{n} \\
                                          \end{bmatrix}=I_{n}$ we have
$$w(A)=\mathcal{D}(0,\max_{\substack{\Xi\in\M_{m,l} \\
              \Xi^{*}\Xi=I_{l}}}\norm{\Xi^{*}A}_{2}),$$
where $\Xi=\begin{bmatrix}
             \xi_{1} & \ldots & \xi_{l} \\
           \end{bmatrix}$.

For a pair of matrices $A,B\in\M_{m,n}$ the numerical
$w_{\norm\cdot}(A,B)$ as it has been presented in \eqref{w3} and
\eqref{w4}, imposes the question ''how $w_{\norm\cdot}(A,B)$ in
\eqref{w4} is independent of B''. An answer is given in the next
proposition.

\begin{prop}
Let $A,B\in\M_{m,n}$ such that $\norm B_{F}\geq 1$. Then
   $$\!\!\!\!\!\!\!\!\!\!\!\!\!\!\!\!\!\!\!\!\!\!\!\!\!\!\!\!\!\!\!\!\!\!\!\!\!\!\!\!\!\!
   \!\!\!\!\!\!\!\!\!\!\!\!\!\!\!\!\!\!\!\!\!\!\!\!\!\!\!\!\!\!\!\!\!\!\!\!\!\!\!\!\!\!\!
   \!\!\!\!\!\!\!\!\!\!\!\!\!\!
    \textbf{\textsc{I.}}
   \bigcup_{\norm B_{F}\geq 1}{w_{\norm{\cdot}_{F}}(A,B)}=\mathcal{D}(0,\norm A_{F}). $$
   \begin{description}
   \item[\,II.]
   \quad\!\!\!If\, $rankB=k$ and \,$\norm\sigma_{F}\geq\sqrt{k}$, where\, the\, vector $\sigma=(\sigma_{1}, \ldots, \sigma_{k})$
   corresponds to the singular values of $B$, then the centers of the discs
   in \eqref{w4}, $\frac{\seq{A,B}}{\norm B_{F}^{2}}\in \mathcal{D}(0,\norm A_{2}).$
   \end{description}
\end{prop}
\begin{proof}
\textbf{I.} Let $z\in \bigcup_{\norm
B_{F}\geq1}{w_{\norm{\cdot}_{F}}(A,B)}$, then there exists a matrix
$B_{0}\in\M_{m,n}$ with $\norm{B_{0}}_{F}\geq1$, such that
$|z-\frac{\seq{A,B_{0}}}{\norm{B_{0}}_{F}^{2}}|\leq \|A-
\frac{\seq{A,B_{0}}}{\norm{B_{0}}_{F}^{2}}B_{0}\|_{F}\sqrt{1-\norm{B_{0}}_{F}^{-2}}$.
Hence,
\begin{equation}\label{pr2}
\abs z\leq\frac{\abs{\seq{A,B_{0}}}}{\norm{B_{0}}_{F}^{2}}+
\|A-\frac{\seq{A,B_{0}}}{\norm{B_{0}}_{F}^{2}}B_{0}\|_{F}\sqrt{1-\norm{B_{0}}_{F}^{-2}}
\end{equation}
and it suffices to show that the right part of \eqref{pr2} is less
than $\norm A_{F}$. In fact, the relationship $(\norm
A_{F}-\abs{\seq{A,B_{0}}})^{2}\geq 0$ is equivalent to
\[
\left( \norm
A_{F}^{2}-\frac{\abs{\seq{A,B_{0}}}^{2}}{\norm{B_{0}}_{F}^{2}}\right)
(1-\norm{B_{0}}_{F}^{-2})\leq\left( \norm
A_{F}-\frac{\abs{\seq{A,B_{0}}}}{\norm{B_{0}}_{F}^{2}}\right)^{2}.
\]
Since, \,\,$\norm
A_{F}^{2}-\frac{\abs{\seq{A,B_{0}}}^{2}}{\norm{B_{0}}_{F}^{2}}=
\|A-\frac{\seq{A,B_{0}}}{\norm{B_{0}}_{F}^{2}}B_{0}\|_{F}^{2}$,\,\,
we have
\[
\frac{\abs{\seq{A,B_{0}}}}{\norm{B_{0}}_{F}^{2}}+\|A-\frac{\seq{A,B_{0}}}{\norm{B_{0}}_{F}^{2}}B_{0}\|_{F}\sqrt{1-\norm{B_{0}}_{F}^{-2}}
\leq\norm A_{F}.
\]
Moreover, if $B_{0}=Ae^{-i\theta}/\norm A_{F}$, $\theta\in[0,2\pi)$,
then $\norm{B_{0}}_{F}=1$ and
\[
\frac{\seq{A,B_{0}}}{\norm{B_{0}}_{F}^{2}}=\norm A_{F}e^{i\theta},
\quad\quad\|A-\frac{\seq{A,B_{0}}}{\norm{B_{0}}_{F}^{2}}B_{0}
\|_{F}=0.
\]
Thus, by \eqref{w4} we have
\[
\abs{z-\norm A_{F}e^{i\theta}}\leq 0\,\, \Rightarrow\,\, z=\norm
A_{F}e^{i\theta}\,\, \Rightarrow\,\,\abs z=\norm A_{F}
\]
thereby, the boundary of $\mathcal{D}(0,\norm A_{F})$ is
attained.\\\\
\textbf{II.} Denoting by $\lambda(\cdot)$ and $\sigma(\cdot)$ the
eigenvalues and singular values of matrices and making use of known
inequalities \cite[p.176,177]{H.J.T} it follows that
\[
\frac{\abs{\seq{A,B}}}{\norm B_{F}^{2}}=
\frac{\abs{tr(B^{*}A)}}{\norm B_{F}^{2}}=
\frac{\abs{\sum{\lambda(B^{*}A)}}}{\norm B_{F}^{2}}
\leq\frac{\sum{\abs{\lambda(B^{*}A)}}}{\norm B_{F}^{2}}
\leq\frac{\sum{\sigma(B^{*}A)}}{\norm B_{F}^{2}}
\]
\begin{equation}\label{pr3}
\!\!\!\!\!\!\!\!\!\!\!\!\!\!\!\!\!\!\!\!\!\!\!\!\!\!\leq\frac{\sum{\sigma(B^{*})\sigma(A)}}{\norm B_{F}^{2}}
\leq\sigma_{\max}(A)\frac{\sum{\sigma(B)}}{\sum{\sigma^{2}(B)}}.
\end{equation}
Since $\norm\sigma_{F}\geq\sqrt{k}$, then
$\sum{\sigma^{2}(B)}=\norm\sigma_{F}^{2}\geq\sqrt{k}\norm\sigma_{F}\geq\seq{\textbf{1},\sigma}=\sum{\sigma(B)}$
and consequently by \eqref{pr3},
$$\frac{\abs{\seq{A,B}}}{\norm B_{F}^{2}}\leq\sigma_{\max}(A)=\norm A_{2}.$$
\end{proof}

The conclusions of proposition 5 strengthen the definition $w(A)$ in \eqref{w2} since the independence
of $w_{\norm\cdot_{F}}(A,B)$ by the matrix $B$ leads to a circular disc.

\begin{prop}
Let $A\in\M_{m,n}$, then
\[
w(A)=\set{\seq{A,B} : B\in\M_{m,n},\,\, rank B=1,\,\, \norm
B_{F}=1}.
\]
\end{prop}
\begin{proof}
Let $z\in w(A)$, then there exist unit vectors
$x\in\Complex^{n}$, $y\in\Complex^{m}$ such that
\[
z=y^{*}Ax=tr(y^{*}Ax)=tr(Axy^{*})=\seq{A,yx^{*}}
\]
Denoting by $B=yx^{*}$, obviously $rankB=1$ and
\[
\norm B_{F}^{2}=tr(B^{*}B)=tr(xy^{*}yx^{*})=tr(xx^{*})=tr(x^{*}x)=1.
\]
Conversely, if $rankB=1$ then $B=yx^{*}$ and evenly
$\seq{A,yx^{*}}=tr(xy^{*}A)=y^{*}Ax$. Since, $1=\norm
B_{F}^{2}=tr(xy^{*}yx^{*})=\norm x_{2}^{2}\norm y_{2}^{2}$, to the case where $x, y$
are not unit vectors, let $\norm y\geq 1$, then
$\norm x\leq 1$ and we verify that the point
$y^{*}Ax=\frac{y^{*}}{\norm y}A\frac{x}{\norm x}$ belongs to $w(A)$.
\end{proof}
\textbf{Example.}
If  $A=\left[\begin{array}{ccc}
         6+i & 0 & 1/2 \\
         -4 & -3-6i & 0 \\
         \end{array}\right]$, Propositions 5 and 6 are illustrated in the next
figure, where the drawing discs $w_{\norm\cdot}(A,B)$ in \eqref{w4},
for six different matrices $B$ with $\norm B_{F}\geq 1$, approximate
the disc $\mathcal{D}(0,\norm A_{F})$. The dashed circle is
$w(A)$ in \eqref{w2}.

\begin{center}                                                 
\includegraphics[width=0.5\textwidth]{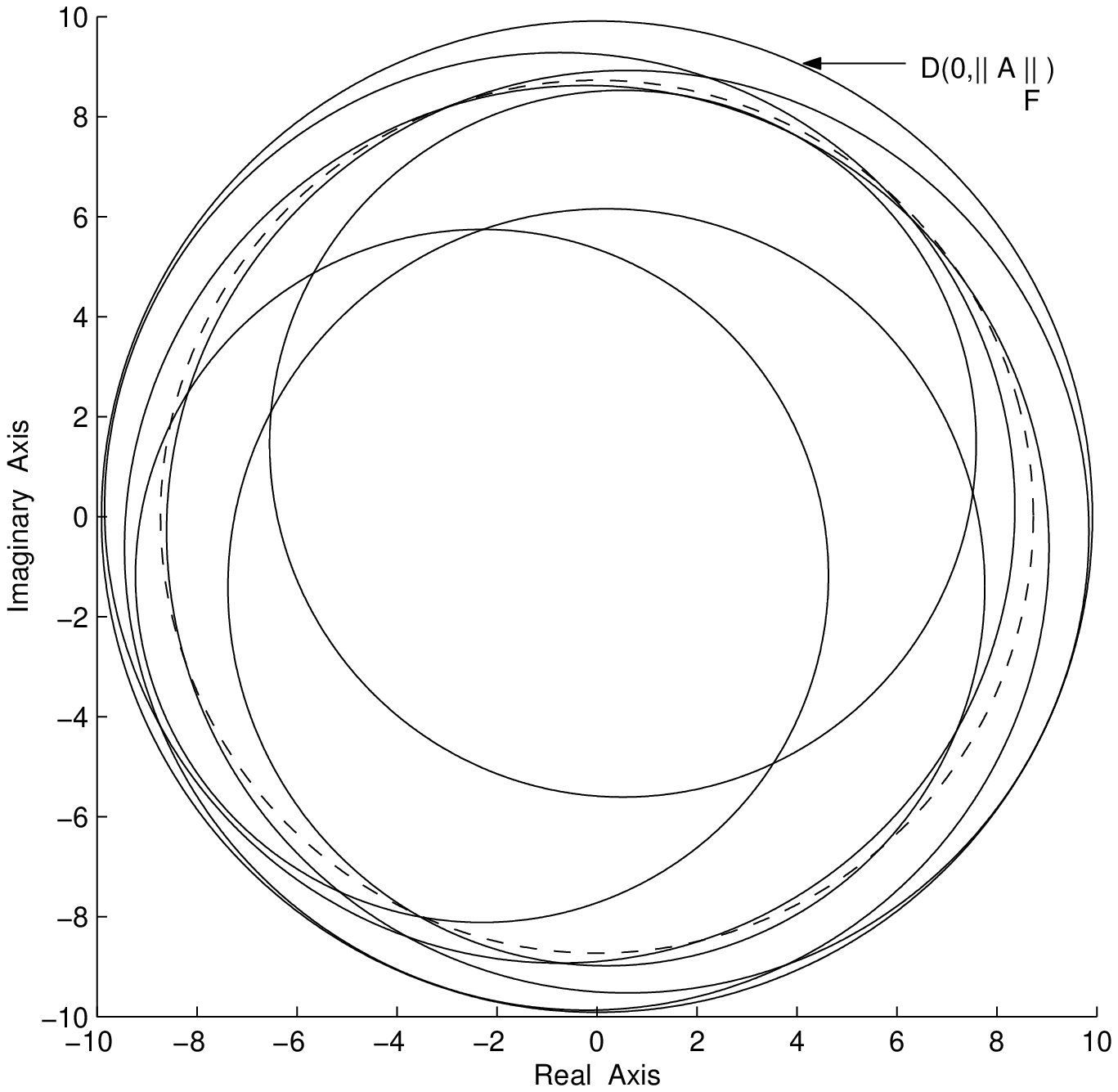}                 
\end{center}                                                   


\section{Properties of $w_{l}(A)$ and $w_{h}(A)$}

In the introduction we have been referred to the numerical ranges
$w_{l}(A)$ and $w_{h}(A)$ for rectangular matrices with respect to unitary $m\times n$ matrix $H$.
Let $A=\left[
                                                  \begin{array}{c}
                                                    A_{1} \\
                                                    A_{2} \\
                                                  \end{array}
                                                \right]$, (or $A=\left[
                                                               \begin{array}{cc}
                                                                 A_{1} & A_{2} \\
                                                               \end{array}
                                                             \right]$),
with $A_{1}$ to be square. Then for $H=\left[
                                         \begin{array}{c}
                                           I \\
                                           0 \\
                                         \end{array}
                                       \right]$, by
\eqref{w5}-\eqref{w6} we have
\[
w_{l}(A)=F(A_{1})\,\,\,\, \textrm{and}\,\,\,\,
w_{h}(A)=F(\left[
                                             \begin{array}{cc}
                                               A & 0_{m\times (m-n)} \\
                                             \end{array}
                                           \right]), \,\textrm{when}\,\,\,\,\, m>n
\]
and by \eqref{w7} we have
\[
w_{l}(A)=F(A_{1})\,\,\,\,\textrm{and}\,\,\,\, w_{h}(A)=F(\left[
                                                   \begin{array}{c}
                                                     A \\
                                                     0_{(n-m)\times n} \\
                                                   \end{array}
                                                 \right]),\, \textrm{when}\,\,\,\,\, m<n.
\]

\begin{prop}
Let the vector $\textbf{a}=\left[
                             \begin{array}{cccc}
                               a_{1} & a_{2} & \ldots & a_{m} \\
                             \end{array}
                           \right]^{T}\in \Complex^{m}$,
then $w_{h}(\textbf{a})$ with respect to $H=\begin{bmatrix}
                                              I_{1} \\
                                              0 \\
                                            \end{bmatrix}$
is the elliptical disc with focal points 0 and
$a_{1}$, the major axis has length $\norm{\textbf{a}}_{2}$ and the minor
axis has length
$\norm{\textbf{b}}_{2}$, where
$\textbf{b}=\left[
         \begin{array}{ccc}
           a_{2} & \ldots & a_{m} \\
         \end{array}
       \right]^{T}$.
\end{prop}
\begin{proof}
By the definition $w_{h}(\textbf{\emph{a}})=F(\left[
                                     \begin{array}{cc}
                                       \textbf{\emph{a}} & 0_{m\times(m-1)} \\
                                     \end{array}
                                   \right])$. If $\textbf{\emph{b}}$ is not
collinear of $\eps_{1}=\left[
                         \begin{array}{cccc}
                           1 & 0 & \ldots & 0 \\
                         \end{array}
                       \right]^{T}\in\Complex^{m-1}$, we consider
the\, Householder \,matrix $H=I_{m-1}-2\frac{uu^{*}}{\norm u^{2}}$,
with $u=\textbf{\emph{b}}-\frac{\norm{\textbf{\emph{b}}}_{2}a_{2}}{\abs{a_{2}}}\eps_{1}$.
Then
\[
\left[
  \begin{array}{cc}
    1 & 0 \\
    0 & H \\
  \end{array}
\right]\left[
         \begin{array}{cc}
           \textbf{\emph{a}} & 0 \\
         \end{array}
       \right]\left[
                \begin{array}{cc}
                  1 & 0 \\
                  0 & H^{*} \\
                \end{array}
              \right]=diag\left(\left[
                                  \begin{array}{cc}
                                    a_{1} & 0 \\
                                    \frac{\norm{\textbf{\emph{b}}}_{2}a_{2}}{\abs{a_{2}}} & 0 \\
                                  \end{array}
                                \right], 0_{m-2}\right).
\]

Hence, \,$F(\left[
                                     \begin{array}{cc}
                                       \textbf{\emph{a}} & 0_{m\times(m-1)} \\
                                     \end{array}
                                   \right])=F(\left[
                                  \begin{array}{cc}
                                    a_{1} & 0 \\
                                    \frac{\norm{\textbf{\emph{b}}}_{2}a_{2}}{\abs{a_{2}}} & 0 \\
                                  \end{array}
                                \right])$ and the\, numerical\, range\, on
the right is the elliptical disc with the aforementioned
 characte\-ristic\, features.
\end{proof}

By Proposition 7, clearly, $w_{h}(\textbf{\emph{a}})=\set{z
: \abs z\leq \norm{\textbf{\emph{b}}}_{2}}$, when $a_{1}=0$. Moreover, if
$\textbf{\emph{a}}\in\M_{1,m}$, it is explicitly viewed that
$w_{h}(\textbf{\emph{a}})$ is the same elliptical disc.
\begin{prop}
Let $m>n$ and $A\in\M_{m,n}$. If $A=\left[
                                       \begin{array}{c}
                                         A_{1} \\
                                         A_{2} \\
                                       \end{array}
                                     \right]$, where $A_{1}$ is the
principal $n\times n$ submatrix of $A$, then
\begin{description}
  \item[ \,\,\,\textbf{I.}]
  $w_{l}(A)\subseteq w_{h}(A)$ for every unitary $H\in\M_{m,n}$.
  \item[\,\,\textbf{II.}]
  $w(A)=\bigcup_{H}{w_{l}(A)}=\bigcup_{H}{w_{h}(A)}.$
  \item[\textbf{III.}]
  $\RE w_{h}(A)=F(\left[
                  \begin{array}{cc}
                    \mathcal{H}(A_{1}) & A_{2}^{*}/2 \\
                    A_{2}/2 & 0_{m-n} \\
                  \end{array}
                \right]$, $\IM w_{h}(A)=F(\left[
                  \begin{array}{cc}
                    \mathcal{S}(A_{1}) & -A_{2}^{*}/2 \\
                    A_{2}/2 & 0_{m-n} \\
                  \end{array}
                \right]$
  with respect to unitary $H=\begin{bmatrix}
                               I_{n} \\
                               0 \\
                             \end{bmatrix}$,
  where $\mathcal{H}(\cdot)$ and $\mathcal{S}(\cdot)$ denote the hermitian and
  skew-hermitian part of matrix, respectively.
  \item[\textbf{IV.}]
  $\sigma(A_{1})\subseteq w_{h}(A)\subseteq w(A)$ with $H=\begin{bmatrix}
                               I_{n} \\
                               0 \\
                             \end{bmatrix}$.
\end{description}
\end{prop}
\begin{proof}
\textbf{I.} Let the unitary matrix $U=\left[
                                         \begin{array}{cc}
                                           H & R \\
                                         \end{array}
                                       \right]\in\M_{m,m}$, where
$H\in\M_{m,n}$. Then
\[
w_{h}(A)=F(AH^{*})=F(U^{*}AH^{*}U)=F(\left[
                                       \begin{array}{cc}
                                         H^{*}A & 0 \\
                                         R^{*}A & 0 \\
                                       \end{array}
                                     \right])
\]
whereupon $w_{l}(A)=F(H^{*}A)\subseteq w_{h}(A)$.

\textbf{II.} Suppose\, $z\in\bigcup w_{l}(A)=\bigcup_{H}{F(H^{*}A)}$, then
for\, a \,$m\times n$ unitary\, matrix \,$H$
\[
\abs z\leq r(H^{*}A)\leq \norm{H^{*}A}_{2}\leq \norm{H^{*}}_{2} \norm A_{2}=\norm A_{2}
\]
where $r(\cdot)$ denotes the numerical radius of matrix. Thereby,
$\bigcup{w_{l}(A)}=\bigcup_{H}{F(H^{*}A)}\subseteq w(A)$.
On the other side, if $z=y^{*}Ax\in w(A)$, then there exists a $m\times n$
unitary matrix $H$ such that $y=Hx$ and $z=x^{*}(H^{*}A)x\in F(H^{*}A)$. The assertion
$\bigcup{w_{h}(A)}=w(A)$ is established similarly.

\textbf{III.} It is enough to confirm that for the $m\times n$ unitary matrix $H=\begin{bmatrix}
                                                                        I_{n} \\
                                                                        0 \\
                                                                      \end{bmatrix}$
$$\RE{w_{h}(A)}=\RE{F(\left[
                                                            \begin{array}{cc}
                                                              A & 0 \\
                                                            \end{array}
                                                          \right])=F(\mathcal{H}(\left[
                                                            \begin{array}{cc}
                                                              A & 0 \\
                                                            \end{array}
                                                          \right]))},$$
where $\mathcal{H}(\cdot)$ denotes the hermitian part of matrix. Similarly,
for
$\IM{w_{h}(A)}$.

\textbf{IV.} We need merely to apply cases (I) and (II) for the $m\times n$ unitary matrix
$H=\begin{bmatrix}
     I_{n} \\
     0     \\
   \end{bmatrix}$.
\end{proof}

By the definitions \eqref{w5},\eqref{w6} or \eqref{w7} it is clear
that the concept of sharp point \cite[p.50]{H.J.T} of $F(AH^{*})$ or
$F(H^{*}A)$ is transferred to the sharp point of $w_{h}(A)$ or
$w_{l}(A)$, respectively. Especially, we note:

\begin{prop}
Let $A\in\M_{m,n}$, $m>n$  and $\lambda_{0} (\neq 0)$ be sharp point of $w_{h}(A)=F(AH^{*})$ for
$H\in\M_{m,n}$, $H^{*}H=I_{n}$. Then
$\lambda_{0}\in \sigma(H^{*}A)$ and is also sharp point of
$w_{l}(A)=F(H^{*}A)$.
\end{prop}
\begin{proof}
For the sharp point
$\lambda_{0}\in\partial w_{h}(A)=\partial F(AH^{*})$ with $H^{*}H=I_{n}$ apparently,
$\lambda_{0}\in\sigma (AH^{*})=\sigma(U^{*}AH^{*}U)=\sigma(H^{*}A)\cup\set 0$,
for the unitary matrix $U=\begin{bmatrix}
                        H & R \\
                      \end{bmatrix}\in\M_{m,m}$, i.e. $\lambda_{0}\in\sigma(H^{*}A)\subseteq
F(H^{*}A)=w_{l}(A)$.

Moreover, for $\lambda_{0}$, according to the definition of sharp point, there exist $\theta_{1}, \theta_{2}\in
[0,2\pi)$, $\theta_{1}<\theta_{2}$ such that
\[
\RE{(e^{i\theta}\lambda_{0})}=\max{\set{{\RE
a: a\in e^{i\theta} w_{h}(A)}}}
\]
for all $\theta\in(\theta_{1}, \theta_{2})$.
Since $w_{h}(A)\supseteq w_{l}(A)$ we have
\[
\RE{(e^{i\theta}\lambda_{0})}=\max_{a\in e^{i\theta} w_{h}(A)}\RE
a\geq \max_{b\in e^{i\theta} w_{l}(A)}{\RE b}
\]
for all $\theta\in (\theta_{1}, \theta_{2})$.

Furthermore, for every $\theta\in (\theta_{1}, \theta_{2})$
\[
\RE{(e^{i\theta}\lambda_{0})}\in\RE{(e^{i\theta}F(H^{*}A))}\leq\max\set{\RE
b : b\in e^{i\theta} F(H^{*}A)}
\]
and \,thus\,
$\RE{(e^{i\theta}\lambda_{0})}=\max\set{\RE
b : b\in e^{i\theta} F(H^{*}A)} $ for\, all $\theta\in(\theta_{1}, \theta_{2})$,\,\, concluding that $\lambda_{0}
(\neq0)$ is sharp point of $F(H^{*}A)=w_{l}(A)$.
\end{proof}

For $m\times n$ unitary matrix $H= \begin{bmatrix} I_{n} \\ 0 \\ \end{bmatrix}$
we may obviously see the following corollary.

\begin{cor}
Let $A_{1}\in\M_{n,n}$ be the principal submatrix of $A\in\M_{m,n}$ and $\lambda_{0} (\neq 0)$
be sharp point of $w_{h}(A)=F(\begin{bmatrix}
                                                                                A & 0 \\
                                                                              \end{bmatrix})$.
Then $\lambda_{0}\in \sigma(A_{1})$ and is also sharp point of
$w_{l}(A)=F(A_{1})$.
\end{cor}

It is noticed here that the converse of Proposition 9 does not hold as it is illustrated
in the next figure. If $A=\begin{bmatrix}
                 1+i & -7 & 0 \\
                 5i & 0.02 & 0 \\
                 0 & 0 & 6-i \\
                 0 & 0 & 0 \\
               \end{bmatrix}$
and $H=\begin{bmatrix}
         0 \\
       I_{3} \\
    \end{bmatrix}$,
$\lambda_{0}=5i$ is sharp point of $w_{l}(A)$ but not of $w_{h}(A)$. Note that by '$*$'
are denoted the eigenvalues $0$ and $5i$ of $AH^{*}$.
\begin{center}
    \includegraphics[width=0.45\textwidth]{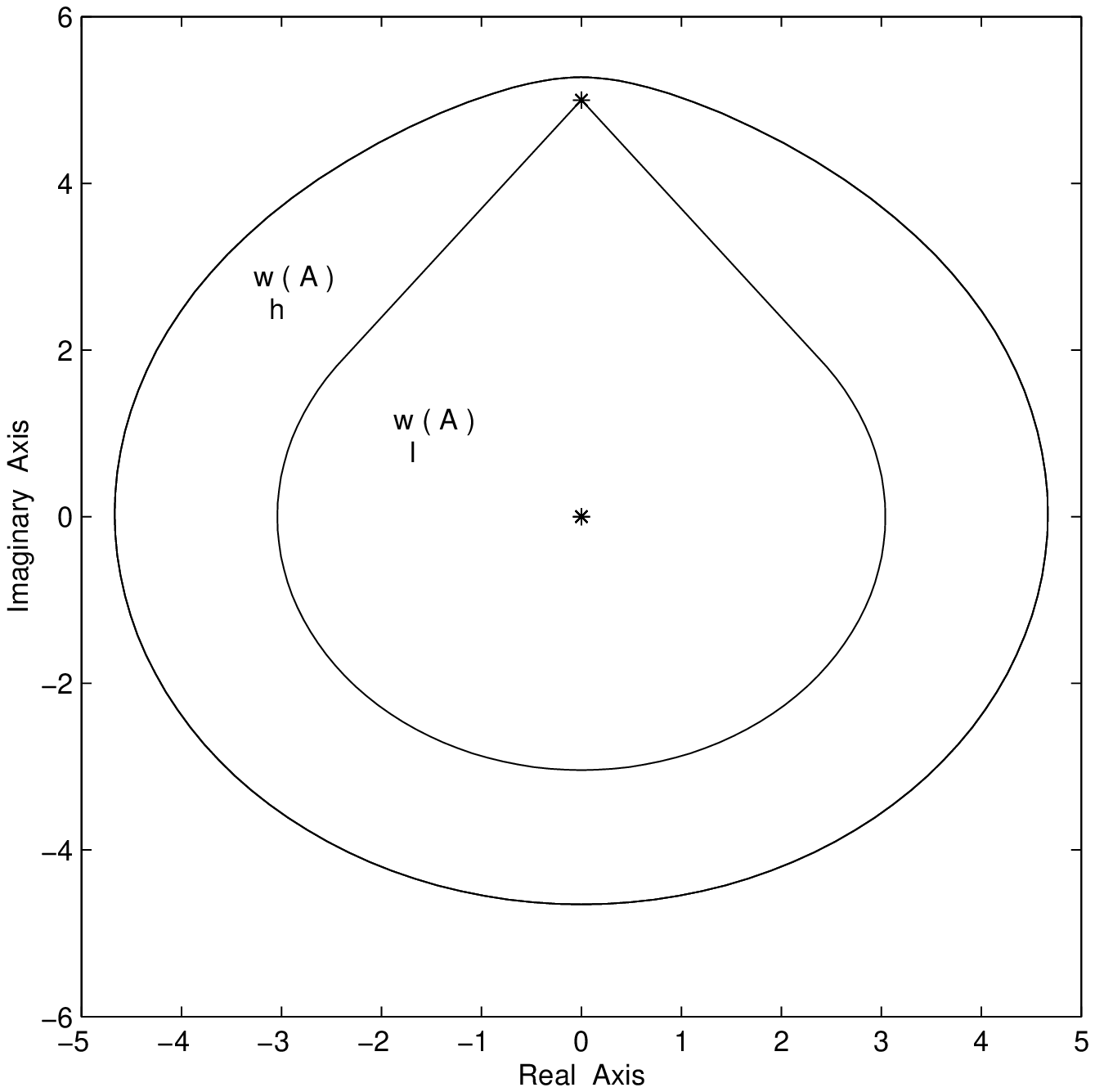}                            
\end{center}


\section{The rank-k numerical range}

In this section, initially, we note the easily confirmed properties
of $\phi_{k}(A)$ in \eqref{w10} :
\[
  \phi_{k}(cA)=c\phi_{k}(A),\,\,\, c \in\Complex\quad \textrm{and} \quad
  \phi_{k}(A^{*})=\overline{\phi_{k}(A)}.
\]
Also, in the next proposition we generalize some necessary and sufficient
conditions \cite{Choi} for $\Lambda_{k}(A)$ in \eqref{w9}, which are extended to $\phi_{k}(A)$.

\begin{prop}
Let $A\in \M_{m,n}$. The next expressions are equivalent.
\begin{description}
  \item[\,\,\, \textbf{I.}]
  $z\in \phi_{k}(A)$.
  \item[\,\,\textbf{II.}]
  There exist subspaces $\J\subseteq\Complex^{m}$ and
  $\K\subseteq\Complex^{n}$ such that $\dim{\J}=$
  $\dim{\K}=k$ and $(A-zS)\K\bot\J$.
  \item[\textbf{III.}]
  There exist orthonormal matrices $M\in\M_{m,k}$ and
  $N\in\M_{n,k}$
  such that $M^{*}AN=zI_{k}$.
  \item[\textbf{IV.}]
  $\seq{Av,u}=z\seq{\tilde{v},
  \tilde{u}}$, where $v=N\tilde{v}$, $u=M\tilde{u}$ and $M, N$ are the matrices
  in \textsc{(III)}.
  \item[\,\,\textbf{V.}]
  There exist subspaces $\Lom\subseteq \Complex^{m}$ and
  $\G\subseteq\Complex^{n}$ of dimension k, where
  $\seq{Av,u}=z\norm v\norm u$, for every $u\in \Lom$ and
  $v\in\G$.
\end{description}
\end{prop}
\begin{proof}
We prove that (I) is equivalent to
(II), (III), (IV) and (V).

\textbf{II.} For $z\in \phi_{k}(A)$, clearly by \eqref{w10}
$P(A-zS)Q=0$. If $\J=Im(P)\subseteq \Complex^{m}$ and $\K
=Im(Q)\subseteq\Complex^{n}$, then $\dim{\J}=\dim{\K}=k$ and for
every $x\in \K$, $y\in\J$, we have
\[
\quad\,\,\,\seq{(A-zS)x,y} = \seq{(A-zS)Qx',Py'} =
\seq{P^{*}(A-zS)Qx', y'}
\]
\[
\!\!=\seq{P(A-zS)Qx',y'} = 0
\]
whereupon $(A-zS)\K\bot\J$. Conversely, by orthogonality we have :
\[
\seq{(A-zS)x,y}=0 \quad\forall\, x\in \mathcal{K}, y\in \mathcal{J}
\,\,\Rightarrow
\]
\[
\seq{(A-zS)Qx',Py'}=0 \quad\forall \,x', y'
\,\,\Rightarrow
\seq{P(A-zS)Qx',y'}=0 \quad\forall\, x', y' \,\,\Rightarrow
\]
\[
\!\!\!\!\!\!\!\!\!\!\!\!\!\!\!\!\!\!\!\!\!
\!\!P(A-zS)Q=0 \Rightarrow
 z\in
\phi_{k}(A).
\]

\textbf{III.} Let\, the matrices $M= \left[
                                     \begin{array}{ccc}
                                       u_{1} & \ldots & u_{k} \\
                                     \end{array}
                                   \right]$
and $N=\left[
         \begin{array}{ccc}
           v_{1} & \ldots & v_{k} \\
         \end{array}
       \right]$,
where \,their \,columns $u_{j}$,\, $v_{i}$ \,constitute \,orthonormal \,bases of
$\J$\, and \,$\K$ \,in (II), respectively. Then, by
statement (II) :
\[
0=\seq{(A-zS)v_{i},u_{j}} = \seq{Av_{i},u_{j}}- z\seq{Sv_{i},u_{j}}.
\]
Denoting by $S=MN^{*}=\sum_{l=1}^{k}{u_{l}v_{l}^{*}}$, we obtain
\[
\seq{Av_{i},u_{j}}=z\langle\sum_{l=1}^{k}{u_{l}v_{l}^{*}v_{i},\,
u_{j}}\rangle=z\sum_{l=1}^{k}{u_{j}^{*}u_{l}v_{l}^{*}v_{i}}=z
\]
for $l=i=j$, and thereby $M^{*}AN=zI_{k}$. For the converse, by the
equation $M^{*}AN=zI_{k}$ with $M^{*}M=N^{*}N=I_{k}$ we have
$\seq{Av_{i},u_{j}}=\delta_{ij}z$, for $i, j=1, \ldots, k$, where
$\delta_{ij}$ is the Kronecker symbol. Hence, $PAQ=zS$, where
$P=MM^{*}$, $Q=NN^{*}$ and $S=MN^{*}$, i.e. $z\in \phi_{k}(A)$.

\textbf{IV.} If $u=\lambda_{1}u_{1}+ \ldots+ \lambda_{k}u_{k}$ and $v=\mu_{1}v_{1}+
\ldots+ \mu_{k}v_{k}$, then by (III):
\[
\seq{Av,u}=u^{*}Av= \left[\begin{array}{ccc}
                      \bar{\lambda}_{1} & \ldots & \bar{\lambda}_{k} \\
                    \end{array}\right]M^{*}AN\left[\begin{array}{c}
                                    \mu_{1} \\
                                    \vdots \\
                                    \mu_{k} \\
                                  \end{array}\right]
\]
\[
\quad\quad=z\left[\begin{array}{ccc}
                      \bar{\lambda}_{1} & \ldots & \bar{\lambda}_{k} \\
                    \end{array}\right]\left[\begin{array}{c}
                                    \mu_{1} \\
                                    \vdots \\
                                    \mu_{k} \\
                                  \end{array}\right]
                                  =z\seq{\tilde{v},\tilde{u}}.
\]

Conversely, by the equation
$\seq{Av,u}=z\seq{\tilde{v},\tilde{u}}$, for $v=v_{i}=Ne_{i}$
and $u=u_{j}=Me_{j}$, where $e_{i}, e_{j}$ are vectors of standard basis of $\Complex^{k}$, we have
\[
\seq{Av_{i},u_{j}}=z\seq{e_{i},e_{j}}\,\,\Rightarrow \,\,u_{j}^{*}Av_{i}=e_{j}^{*}M^{*}ANe_{i}=\delta_{ij}z\,\,\,;\, i, j= 1, \ldots, k
\]
or equivalently $M^{*}AN=zI_{k}$, i.e. $z\in\phi_{k}(A)$.

\textbf{V.} Let $z\in\phi_{k}(A)$ and $u\in span\set{u_{2}, \ldots,
u_{k}}^{\bot}$, $v\in span\set{v_{2}, \ldots, v_{k}}^{\bot}$, where $u_{j}\in\Complex^{m}$,
$v_{i}\in\Complex^{n}$ are orthonormal vectors.
Denoting by
\[
M=\left[\begin{array}{cccc}
      \frac{u}{\norm u} & u_{2} & \ldots & u_{k} \\
    \end{array}\right],\,\, N=\left[\begin{array}{cccc}
      \frac{v}{\norm v} & v_{2} & \ldots & v_{k} \\
    \end{array}\right]
\]
clearly $M^{*}M=N^{*}N=I_{k}$. By statement (II) and for $P=MM^{*}$, $Q=NN^{*}$ and
$S=MN^{*}$ we have $(A-zS)\G\bot\Lom$, where $\G=Im(Q)$, $\Lom=Im(P)$. Thus, we obtain
\[
\seq{(A-zS)v,u}=0 \Rightarrow \seq{Av,u}=z\seq{Sv,u}=z\seq{MN^{*}v,
u}=z\seq{N^{*}v,M^{*}u}
\]
\[
\quad\quad\quad\quad\quad=z\frac{v^{*}v}{\norm v}\frac{u^{*}u}{\norm
u}=z\norm v\norm u.
\]
The converse is received trivially, completing the proof.
\end{proof}

\begin{prop}
The rank-k numerical range $\phi_{k}(A)$ for a rectangular
matrix $A\in\M_{m,n}$ satisfies the relationship
\[
w(A)=\phi_{1}(A)\supseteq \phi_{2}(A)\supseteq \ldots
\supseteq \phi_{\tau}(A)
\]
where $\tau=\min{\set{m,n}}$.
\end{prop}
\begin{proof}
Let $z\in \phi_{k}(A)$ and $u$, $v$ are unit vectors of $\Complex^{m}$ and $\Complex^{n}$. Then, by
proposition  11(V), we derive $\seq{Av,u}=z$, i.e. $z\in
\phi_{1}(A)=w(A)$. Hence,
$\phi_{k}(A)\subseteq\phi_{1}(A)$ for every k. Besides, if
$z\in\phi_{k}(A)$, by Proposition 11(III) we have
$M^{*}AN=zI_{k}$, where $M=\left[
                                                             \begin{array}{ccc}
                                                               u_{1}& \ldots & u_{k} \\
                                                             \end{array}
                                                           \right]=\left[
                                                                     \begin{array}{cc}
                                                                       M_{1} & u_{k} \\
                                                                     \end{array}
                                                                   \right]
$ and $N=\left[\begin{array}{ccc}
                             v_{1}& \ldots & v_{k} \\
\end{array}\right]=\left[
                     \begin{array}{cc}
                       N_{1} & v_{k} \\
                     \end{array}
                   \right]
$ are orthonormal matrices. Then $M_{1}^{*}AN_{1}=zI_{k-1}$, i.e.
$z\in\phi_{k-1}(A)$, concluding that
$\phi_{k}(A)\subseteq\phi_{k-1}(A)$ for $k=2, \ldots,
\tau$, where  $\tau=\min{\set{m,n}}$.
\end{proof}

Following, we present some additional properties :

\begin{prop}
Let $A\in \M_{m,n}$, then for $\phi_{k}(A)$ in \eqref{w10}, holds:
\begin{description}
  \item[ \,\,\,I.]
  $\,\phi_{k}(U^{*}AV)=\phi_{k}(A)$, where $U\in\M_{m,m}$
  and $V\in\M_{n,n}$ are unitary matrices.
  \item[ II.]
  $\,\phi_{k}(A)=\phi_{k}(e^{i\theta}A)$ for every $\theta\in[0,2\pi)$.
  \item[III.]
  If $z\in\phi_{k}(A)$, then  $\RE z \in\Lambda_{k}(\begin{bmatrix}
                                                             0 & A \\
                                                             A^{*} & 0 \\
                                                           \end{bmatrix})=[-\sigma_{k}, \sigma_{k}]$
 and $\IM z\in\Lambda_{k}(\begin{bmatrix}
                                                                                    0 & -iA \\
                                                                                    iA^{*} & 0 \\
                                                                                  \end{bmatrix})=[-i\sigma_{k}, i\sigma_{k}]$,
 where $\sigma_{1}\geq\sigma_{2}\geq\ldots\geq\sigma_{q}>0$ denote the decreasingly ordered
singular values of $A$, counting multiplicities.
\end{description}
\end{prop}
\begin{proof}
\textbf{I.} Let $z\in\phi_{k}(U^{*}AV)$, then for suitable
unitary matrices $M\in\M_{m,k}$ and $N\in\M_{n,k}$ we have
$M^{*}U^{*}AVN=zI_{k}\,\,\,\Rightarrow\,\,\,(UM)^{*}A(VN)=zI_{k}$
i.e. $z\in\phi_{k}(A)$. Thus
$\phi_{k}(U^{*}AV)\subseteq\phi_{k}(A)$.

Conversely, if
$z\in\phi_{k}(A)$, then $R^{*}AT=zI_{k}$, where
$R\in\M_{m,k}$ and $T\in\M_{n,k}$ are unitary. Clearly we can write $R=UM$ and
$T=VN$, where $U$ and $V$ are defined by orthonormal bases of
$\Complex^{m}$ and $\Complex^{n}$, respectively. Therefore,
$M^{*}(U^{*}AV)N=zI_{k}$, i.e. $z\in\phi_{k}(U^{*}AV)$.

\textbf{II.}
Assume $M\in\M_{m,k}$ and $N\in\M_{n,k}$ such that $M^{*}M=N^{*}N=I_{k}$, then
\begin{eqnarray*}
   \phi_{k}(A) & = & \set{z\in\Complex : M^{*}AN=zI_{k}} \\
               & = & \{z\in\Complex : (M^{*}e^{-i\theta})(e^{i\theta}A)N=zI_{k}\} \\
               & = & \{z\in\Complex : (e^{i\theta}M)^{*}(e^{i\theta}A)N=zI_{k}\} \\
               & = & \{z\in\Complex : M_{1}^{*}(e^{i\theta}A)N=zI_{k}\}=\phi_{k}(e^{i\theta}A)
\end{eqnarray*}
for every $\theta\in[0,2\pi)$, since $M_{1}^{*}M_{1}=M^{*}M=I_{k}$. That is, the set $\phi_{k}(A)$
is circular.

\textbf{III.} By \eqref{w10}, let $PAQ=zS$, where $P=MM^{*}$, $Q=NN^{*}$
and $S=MN^{*}$. Then, $M^{*}AN=zI_{k}$ and consequently
\begin{equation}\label{pr12}
    (\RE z)I_{k}=\frac{1}{2}\begin{bmatrix}
                              M^{*} & N^{*} \\
                            \end{bmatrix}
                            \begin{bmatrix}
                              0 & A \\
                              A^{*} & 0 \\
                            \end{bmatrix}
                            \begin{bmatrix}
                              M \\
                              N \\
                            \end{bmatrix}
\end{equation}
Denoting by $T=\frac{1}{\sqrt{2}}\begin{bmatrix}
                                   M \\
                                   N \\
                                 \end{bmatrix}\in\M_{(m+n),k}$,
then $T^{*}T=\frac{1}{2}(M^{*}M+N^{*}N)=I_{k}$ and the $(m+n)\times (m+n)$
matrix $G=TT^{*}=\frac{1}{2}\begin{bmatrix}
                   P & S \\
                   S^{*} & Q \\
                 \end{bmatrix}$ is rank-$k$ orthogonal projector,
because rank$T=k$ and
\[
G^{2}=\frac{1}{4}\begin{bmatrix}
                   P^{2}+SS^{*} & PS+SQ \\
                   S^{*}P+QS^{*} & S^{*}S+Q^{2} \\
                 \end{bmatrix}=\frac{1}{4}
                 \begin{bmatrix}
                 P+P & 2S \\
                 2S^{*} & Q+Q \\
                 \end{bmatrix}=G.
\]
Thus, by \eqref{pr12} we obtain $(\RE z)G=G\begin{bmatrix}
                                   0 & A \\
                                   A^{*} & 0 \\
                                 \end{bmatrix}G$, i.e. $\RE z\in \Lambda_{k}(\begin{bmatrix}
                                                                               0 & A \\
                                                                               A^{*} & 0 \\
                                                                             \end{bmatrix})$.
Similarly, we derive \,$\IM z\in\Lambda_{k}(\begin{bmatrix}
                                  0 & -iA \\
                                  iA^{*} & 0 \\
                                \end{bmatrix})$. Moreover, due to the $(m+n)\times$ $(m+n)$ hermitian matrix  $\begin{bmatrix}
                                                                               0 & A \\
                                                                               A^{*} & 0 \\
                                                                             \end{bmatrix}$
having eigenvalues $-\sigma_{1}\leq-\sigma_{2}\leq\ldots\leq-\sigma_{q}<0<\sigma_{q}\leq\ldots\leq\sigma_{2}\leq\sigma_{1}$,
\cite{H.J.} and the multiplicity\, of $\lambda=0$\, being equal to\, $m+n-2q$,
we \,\,\,verify \cite[Th. 2.4]{C-K-Z} that $\Lambda_{k}(\begin{bmatrix}
                                                                                     0 & A \\
                                                                                     A^{*} & 0 \\
                                                                                   \end{bmatrix})=[-\sigma_{k}, \sigma_{k}]$
and $\Lambda_{k}(\begin{bmatrix}
                                                                                     0 & -iA \\
                                                                                     iA^{*} & 0 \\
                                                                                   \end{bmatrix})=[-i\sigma_{k},i\sigma_{k}]$.
\end{proof}
A more precise description of $\phi_{k}(A)$ is given in the next proposition.
\begin{prop}
Let $A\in\M_{m,n}$ and $\sigma_{1}\geq\sigma_{2}\geq\ldots\geq\sigma_{\min\set{m,n}}$ be its singular values.
\begin{description}
  \item[\,\,\,\,\,I.] If for the index $k$, $\max\set{\frac{m}{2},\frac{n}{2}}<k\leq \frac{m+n+1}{3}$, then $\phi_{k}(A)$
  is equal to the\, ring\, $\mathcal{R}(0;\sigma_{m+n-2k+1}, \sigma_{k})$.
  \item[\,\,II.] If $k>\frac{m+n+1}{3}$, then $\phi_{k}(A)$ is the empty set.
  \item[III.] If $k\leq\max\set{\frac{m}{2},\frac{n}{2}}$, then $\phi_{k}(A)$ is identified with the circular disc
  $\mathcal{D}(0,\sigma_{k})$.
\end{description}
\end{prop}
\begin{proof}
\textbf{I.}
Consider that $A=U\Sigma V^{*}$ is the singular value decomposition of $A$, then by Proposition 13(I),
$\phi_{k}(A)=\phi_{k}(\Sigma)$. If $z\in\phi_{k}(\Sigma)$, then for suitable $m\times k$ and $n\times k$ unitary matrices
$M$ and $N$ we have $zI_{k}=M^{*}\Sigma N$. Denoting by $\tilde{U}=\begin{bmatrix}
                                                                      M & M_{1} \\
                                                                    \end{bmatrix}$
and $\tilde{V}=\begin{bmatrix}
                                                                      N & N_{1} \\
                                                                    \end{bmatrix}$
the augmented unitary square matrices, then the singular values of matrix
\begin{equation}\label{pr42}
\tilde{U}^{*}\Sigma\tilde{V}=\begin{bmatrix}
                                M^{*}\Sigma N & M^{*}\Sigma N_{1} \\
                                M_{1}^{*}\Sigma N & M_{1}^{*}\Sigma N_{1} \\
                                \end{bmatrix}
\end{equation}
are also $\sigma_{1}\geq\sigma_{2}\geq\ldots\geq\sigma_{\min\set{m,n}}$ and the singular values of
submatrix $M^{*}\Sigma N=zI_{k}$ are equal to $\beta_{1}=\beta_{2}=\ldots=\beta_{k}=\abs z$. Thus, by
Th.1 in \cite{Thompson}, we have
\begin{equation}\label{pr43}
\begin{array}{ccc}
 \sigma_{i} & \geq & \!\!\!\!\!\!\!\!\!\!\!\!\!\!\!\!\!\!\!\!\!\!\!\!\!\!\!\!\!\!\!\!\!\!\!\!\!\!\!\!\!\!\!\!\!\!\!\!\!\!\!\!\!
 \beta_{i}=\abs z\,\,,\quad\quad for \,\,\, i=1, \ldots, k,  \\
 \beta_{i} & \geq & \sigma_{i+m+n-2k}\,\,,\,\,\,\,\, for \,\,\,i=1, \ldots, \min\set{2k-m, 2k-n}.
\end{array}
\end{equation}
Since $k\leq\frac{m+n+1}{3}$, then clearly $\sigma_{m+n-2k+1}\leq\sigma_{k}$. The validity of all
inequa\-lities \eqref{pr43} confirms $\sigma_{m+n-2k+1}\leq\abs z\leq\sigma_{k}$ and by the circular property
of $\phi_{k}(\Sigma)$ in Prop. 13(II) we have that $z$ belongs to the ring $\mathcal{R}(0;\sigma_{m+n-2k+1},\sigma_{k})$.

Conversely, if $z\in\mathcal{R}(0;\sigma_{m+n-2k+1},\sigma_{k})$, then
\[
\sigma_{\min\set{m,n}}\leq\ldots\leq\sigma_{m+n-2k+1}\leq \abs z\leq\sigma_{k}\leq\sigma_{k-1}\leq\ldots\leq\sigma_{1}
\]
and by Th.2 in \cite{Thompson}, we have that there exist $m\times k$ and $n\times k$ unitary
$M\in\M_{m,k}$ and $N\in\M_{n,k}$ such that $\beta_{1}=\ldots=\beta_{k}=\abs z$ are the singular
values of the submatrix $M^{*}\Sigma N$ in \eqref{pr42}. Due to the singular values of $zI_{k}$
and $M^{*}\Sigma N$ being identified, the matrices are related by the equation
\[
  W_{1}(zI_{k})W_{2}^{*}=M^{*}\Sigma N
\]
where $W_{1}$, $W_{2}$ are $k\times k$ unitary matrices. Hence, we have
$(MW_{1})^{*}\Sigma(NW_{2})=$ $zI_{k}$, yielding that $z\in\phi_{k}(\Sigma)$.

Note that, for $k=\frac{m+n+1}{3}\,\,\Rightarrow\,\,\sigma_{k}=\sigma_{m+n-2k+1}$, i.e. the ring
is dege\-nerated to the circle $\set{z: \abs z=\sigma_{k}}$.

\textbf{II.}
If $k>\frac{m+n+1}{3}$, then $\sigma_{k}<\sigma_{m+n-2k+1}$ and should be
$z\in\set{z: \abs z\leq\sigma_{k}}\cap\set{z: \abs z\geq\sigma_{m+n-2k+1}}=\emptyset$. Therefore,
$\phi_{k}(A)=\emptyset$.

\textbf{III.}
To the case $k\leq\max\set{\frac{m}{2},\frac{n}{2}}$, obviously $\min\set{2k-m, 2k-n}\leq 0$ and then
only inequalities $\sigma_{i}\geq\beta_{i}=\abs z$ for $i=1, \ldots, k$ are valid, establishing
$\phi_{k}(A)=\mathcal{D}(0,\sigma_{k})$.
\end{proof}

\begin{cor}
Let $A\in\M_{m,n}$ and $\sigma_{1}\geq\ldots\geq\sigma_{\min\set{m,n}}$ be its singular values.
If\,\, $\max\set{\frac{m}{2},\frac{n}{2}}<k\leq \frac{m+n+1}{3}$ and $\sigma_{r}=0$, for
$k\leq r\leq m+n-2k+1$, then $\phi_{k}(A)$
coincides with the circular disc $\mathcal{D}(0,\sigma_{k})$.
\end{cor}
\begin{proof}
Apparently, by Proposition 14(I), $\phi_{k}(A)=\mathcal{R}(0;\sigma_{m+n-2k+1},\sigma_{k})$.
Since index $r$ satisfies $k< r\leq m+n-2k+1$, we have $\sigma_{k}\geq 0\geq\sigma_{m+n-2k+1}$
and then $\phi_{k}(A)=\mathcal{D}(0,\sigma_{k})$.
To the case $k=r$, $\sigma_{k}=\sigma_{r}=0$ and $\phi_{k}(A)$ is degenerated to the origin.
\end{proof}

We remark here that if $\norm A_{2}=\sigma_{1}$ with multiplicity $k$, as it is stated in Corollary 15,
and $\sigma_{l}=0$ for $k< l\leq m+n-2k+1$, then $\phi_{k}(A)=\mathcal{D}(0,\sigma_{1})$.
The boundary points of this disc are reached, using the eigenvectors of $A^{*}A$ corresponding to $\sigma_{1}^{2}$.

\begin{prop}
Let the matrix $A\in\M_{m,n}$. If $\Lom$ is $(m-k+1)$-dimensional subspace of
$\Complex^{m}$ and $\G$ is $(n-k+1)$-dimensional subspace of
 $\Complex^{n}$, then for any positive integer $k\geq 1$
\[
\phi_{k}(A)\subseteq\bigcap_{\Lom}w(P_{\Lom}A) \quad and\quad
\phi_{k}(A)\subseteq\bigcap_{\G}w(AQ_{\G})
\]
where the numerical range $w(\cdot)$ has been defined in \eqref{w2} and $P_{\Lom}, Q_{\G}$ are
orthogonal projectors onto $\Lom$ and $\G$, respectively.
\end{prop}
\begin{proof}
Assume $z\in\phi_{k}(A)$. By Proposition 11(V) there
exist subspaces $\Lom'$ and $\G'$ of $\Complex^{m}$ and
$\Complex^{n}$, respectively, with $\dim\Lom'=\dim\G'=k$, such that
$z=\seq{Av,u}$ for unit vectors $v\in\G',
u\in\Lom'$. Then, following the arguments in \cite{Choi},
for a unit vector $\tilde{u}\in\Lom\cap\Lom'$ we readily see that
$z=\seq{Av,\tilde{u}}=\seq{Av,P_{\Lom}\tilde{u}}=\seq{P_{\Lom}^{*}Av,\tilde{u}}\in w(P_{\Lom}A)$, where
$P_{\Lom}$ is orthogonal projector of $\Complex^{m}$ onto $\Lom$.
Hence, $\phi_{k}(A)\subseteq\bigcap_{\Lom}\set{w(P_{\Lom}A) :
\,P_{\Lom}\,\,orthogonal \,\,projector\,\,onto\,\Lom}$.

Similarly, considering the subspace $\G$ of dimension $n-k+1$ we conclude the second inclusion.
\end{proof}

\textbf{Remark.}
It is worth noticing, finally, the containment
$$\mathcal{D}(0,\sigma_{k}(A))\subseteq\bigcap_{\dim{\mathcal{G}}=n-k+1}{w(AQ_{\mathcal{G}})}
=\mathcal{D}(0,\min_{\mathcal{G}}{\norm{AQ_{\mathcal{G}}}_{2}}),$$
since \cite[p.148]{H.J.T}
$$\min_{\mathcal{G}\subseteq\Complex^{n}}{\norm{AQ_{\mathcal{G}}}_{2}}=
\min_{\mathcal{G}}{\max{\set{\norm{AQ_{\mathcal{G}}x}_{2}:
x\in\Complex^{n}, \norm x_{2}=1}}}$$
$$\quad\quad\quad\quad\quad
\quad\quad\quad\quad\,\,\,\,\geq\min_{\mathcal{G}}{\max{\set{\norm{AQ_{\mathcal{G}}x}_{2}:
x\in \mathcal{G}, \norm x_{2}=1}}}=\sigma_{k}(A).$$

\bibliographystyle{amsplain}
\bibliography{xbib}

\end{document}